\documentclass{article}
\title{A new proof of finitary isomorphism for Markov chains}
\author{Yinon Spinka\thanks{Tel Aviv University. Research supported in part by ISF grant 1361/22}}
\date{June 2025}

\usepackage{amsmath,amsfonts,amssymb,amsthm,graphics,enumerate,mathtools}
\usepackage{hyperref}
\hypersetup{
    colorlinks=true,
    linkcolor=blue,
    citecolor=red,
    urlcolor=blue,
    pdfborder={0 0 0}
}    

\usepackage{graphicx,xcolor,bbm,fullpage}

\usepackage{cleveref}
  \crefname{theorem}{Theorem}{Theorems}
  \crefname{thm}{Theorem}{Theorems}
  \crefname{lemma}{Lemma}{Lemmas}
  \crefname{lem}{Lemma}{Lemmas}
  \crefname{remark}{Remark}{Remarks}
  \crefname{prop}{Proposition}{Propositions}
  \crefname{proposition}{Proposition}{Propositions}
  \crefname{notation}{Notation}{Notations}
  \crefname{claim}{Claim}{Claims}
  \crefname{observation}{Observation}{Observations}
  \crefname{defn}{Definition}{Definitions}
  \crefname{corollary}{Corollary}{Corollaries}
  \crefname{section}{Section}{Sections}
  \crefname{figure}{Figure}{Figures}
  \crefname{exercise}{Exercise}{Exercises}
  \crefname{example}{Example}{Examples}
  \crefname{assumption}{Assumption}{Assumptions}

\newtheorem{thm}{Theorem}[section]

\newtheorem{lemma}[thm]{Lemma}

\newtheorem{prop}[thm]{Proposition}

\numberwithin{equation}{section}

\theoremstyle{definition}

\def\cF{\mathcal{F}}

\def\P{\mathbb{P}}

\def\Z{\mathbb{Z}}
\def\N{\mathbb{N}}

\newcommand{\1}{\mathbf{1}}
\def\eps{\varepsilon}

\begin{document}
\maketitle

\begin{abstract}
We give a new proof of a result of Rudolph stating that a countable-state mixing Markov chain with exponential return times is finitarily isomorphic to an IID process. Besides being short and direct, our proof has the added benefit of working for processes of finite or infinite entropy.
\end{abstract}

\section{Introduction}

Let $X=(X_n)_{n \in \Z}$ and $Y=(Y_n)_{n \in \Z}$ be countable-state (translation invariant) processes.
We say that $Y$ is a \textbf{factor} of $X$ if it can be expressed as an equivariant function of $X$, i.e., if $Y$ has the same distribution as $F(X)$, for some measurable function $F$ such that $F(X+n)=F(X)+n$ almost surely for all $n \in \Z$. The factor is \textbf{finitary} if $F(X)_0$ is almost surely determined by a random, but finite, portion of $X$. More precisely, if there exists a stopping time $\tau$ with respect to the filtration $\cF_n := (X_i)_{|i| \le n}$, such that $F(X)_0$ is measurable with respect to $\cF_\tau$. Equivalently, $F$ is finitary if its restriction to a set of full measure (with respect to $X$) is continuous. A factor that is almost everywhere invertible is called an \textbf{isomorphism}. We say that $X$ and $Y$ are \textbf{finitarily isomorphic} if there exists an isomorphism $F$ such that both $F$ and $F^{-1}$ are finitary.


Let $X=(X_n)_{n \in \Z}$ be a mixing process taking values in a countable space $A$.
We say that a state $a \in A$ (that occurs with positive probability) is a \textbf{renewal state} for $X$ if $(X_n)_{n<0}$ and $(X_n)_{n>0}$ are conditionally independent given $X_0=a$. Thus, $X$ is a Markov chain if and only if every state is a renewal state. We say that $X$ is a \textbf{renewal process} if it has a renewal state. By the \textbf{return time} of a renewal state $s$, we mean an $\N$-valued random variable $T$ with distribution given by
\[ \P(T=n)=\P(X_1,\dots,X_{n-1} \neq s,~X_n=s \mid X_0=s) .\] By exponential return time, we mean that $T$ has exponential tail, i.e., $\P(T>n) \le Ce^{-cn}$ for some $C,c>0$ and all $n \ge 0$.

\begin{thm}\label{thm:renewal}
    Let $X$ be a countable-state mixing renewal process having a renewal state with exponential return time. Then $X$ is finitarily isomorphic to an IID process.
\end{thm}

\noindent\textbf{Background and discussion.}
The problem of classifying processes up to isomorphisms is a fundamental one in ergodic theory and dynamical systems. The most basic processes are finite-valued IID process, also known as Bernoulli schemes. It was an open problem for several decades to determine whether an IID process whose marginals are uniform on two symbols is isomorphic to one whose marginals are uniform on three symbols, until in the late 1950s, Kolmogorov introduced the notion of entropy for dynamical systems, and showed that it is an isomorphism invariant. A decade later, Ornstein~\cite{ornstein1970bernoulli} proved that it is a complete isomorphism invariant for the class of finite-valued IID processes, meaning that any two such processes are isomorphic if and only if they have the same entropy. Prior to this result, Meshalkin~\cite{meshalkin1959case} gave interesting examples of isomorphisms between certain IID processes, with the added benefit that the constructions produced finitary isomorphisms.
In the late 1970s, Keane and Smorodinsky~\cite{keane1979bernoulli} showed that any two finite-valued IID processes of equal entropy are finitarily isomorphic, and shortly after extended this to the class of finite-state mixing Markov chains~\cite{keane1979finitary}.
For countable-valued processes, certain aspects of the problem may become more difficult.
For countable-valued IID processes, that entropy is a complete finitary isomorphism invariant was shown by Meyerovitch and the author~\cite{meyerovitch2024entropy} (the case of infinite entropy was previously established by Petit~\cite{petit1982deux}).
For countable-state mixing Markov chains, while entropy is a complete isomorphism invariant, it is not a complete finitary isomorphism invariant. Indeed, it is not hard to see that in order for such a process to be a finitary factor of an IID process, it must have exponential return times. Rudolph~\cite{rudolph1982mixing} showed that a finite-entropy countable-state mixing Markov chain with exponential return times is finitarily isomorphic to an IID process. This result relies on previous work of the same author~\cite{rudolph1981characterization} in which a rather complicated characterization of those processes which are finitarily isomorphic to an IID process. The case of infinite entropy is left open by Rudolph, though he mentions that his work and Petit's work together indicate that the result should hold in this case too.

Our \cref{thm:renewal} clearly implies that a countable-state mixing Markov chain with exponential return times is finitarily isomorphic to an IID process. In fact, it is not hard to see that a renewal process as in the theorem is finitarily isomorphic to such a Markov chain. Indeed, this can be done by recording at each point the history since the last occurrence of the renewal state. In this way, the theorem is easily seen to be equivalent to its version for Markov chains, which is Rudolph's result (in the finite-entropy case). Note, however, that the obtained Markov chain may have a countable state space, even when the renewal process has only finitely many states. While Keane and Smorodinsky's proof for finite-state Markov chains is direct, the only existing proof of the result for renewal processes goes through Rudolph's heavy machinery and has no ``direct proof'' to date (see~\cite{shea2009finitary} for partial progress in this direction). We therefore hope that our much simpler and more direct approach to proving \cref{thm:renewal} is beneficial.

\section{Proof}

We split the proof of the theorem into two steps, which we state as propositions.
We say that an $\N$-valued random variable $T$ has \textbf{regular exponential tail} if there exist $C>0$ and $b>c>0$ such that
\[ \P(T=n) = Ce^{-cn} \pm O(e^{-bn}) \qquad\text{as }n \to \infty.\]
We say that $T$ has \textbf{semi-regular exponential tail} if it is unbounded and
\[ \liminf_{n \to \infty} \P(T=n \mid T \ge n)>0 \qquad\text{and}\qquad \limsup_{n \to \infty} \P(T=n \mid T \ge n)<1 .\]

\begin{prop}\label{lem:fin-iso1}
    Let $X$ be a countable-state mixing renewal process having a renewal state with exponential return time. Then $X$ is finitarily isomorphic to a finite-state renewal process $Y$ having a renewal state whose return time has regular exponential tail.
\end{prop}

\begin{prop}\label{lem:fin-iso2}
    Let $Y$ be a finite-state ergodic renewal process having a renewal state whose return time has semi-regular exponential tail. Then $Y$ is finitarily isomorphic to an IID process.
\end{prop}

\cref{thm:renewal} follows immediately from these propositions.
The proofs of both propositions rely heavily on the fact that if two renewal processes have renewal states with the same distribution, then the two processes are finitarily isomorphic. As noted in~\cite{keane1979finitary}, for finite-valued processes, this follows from the marker-filler methods of Keane and Smorodinsky~\cite{keane1979bernoulli}. For countable-valued processes, one may appeal to a result of Meyerovitch and the author~\cite[Theorem 3.2]{meyerovitch2024entropy} (see also Lemma 3.5 there). We use this fact repeatedly in the proofs. By the distribution of a state $s$ in a process $X$, we mean the distribution of the process $(\1_{\{X_n=s\}})_{n \in \Z}$.

We shall also use some fairly standard facts and terminology which we briefly recall. By independent splitting of a state in some process, we mean that each occurrence of this state is independently replaced by a sample from some given distribution on a set of new symbols. Such an operation can increase the entropy of the process by any prescribed amount, but does not effect the distributions or renewal properties of other states. A dual operation is that of collapsing a given set of states into a single state. The $k$-stringing of a process $X$ is the process $X^k := ((X_n,X_{n+1},\dots,X_{n+k-1}))_{n \in \Z}$, which is trivially finitarily isomorphic to $X$. Also, if $s$ a renewal state in $X$, then any word of length $k$ containing $s$ in one of its coordinates is a renewal state in $X^k$. Finally, if some renewal state has exponential return time in $X$, then every renewal state has exponential return time in $X$, and also in $X^k$.

For the first proposition, we shall also require the following result by Angel and the author~\cite{angel2021markov}.
Let $T$ be an $\N$-valued random variable, and let $T_1,T_2,\dots$ be independent copies of $T$.
Let $\mu \in (0,1)$ and let $N \sim \text{Geom}(\mu)$ be independent of $\{T_n\}_n$, with the convention that $N$ takes values in the positive integers and $\P(N=n)=\mu (1-\mu)^{n-1}$ for $n \ge 1$. Define
\[ T^*_\mu := T_1 + \cdots + T_N .\]

\begin{lemma}[\cite{angel2021markov}]\label{lem:exp-geom}
Let $T$ be an $\N$-valued random variable with exponential tail and whose support is not contained in any proper subgroup of $\Z$.
Then $T^*_\mu$ has regular exponential tail for all sufficiently small $\mu>0$.
\end{lemma}

\begin{proof}[Proof of \cref{lem:fin-iso1}]
    Let $X$ be a countable-valued mixing renewal process having a renewal state $s_1$ with exponential return time. Assume first that $X$ has a second renewal state $s_2$, and that the entropy of the process $X' := ((\1_{\{X_n=s_1\}},\1_{\{X_n=s_2\}}))_{n \in \Z}$ is strictly less than that of $X$. We first show that, in this case, $X$ is finitarily isomorphic to a finite-state renewal process having a renewal state whose return time has regular exponential tail. We later address the general case.
    
    Let $W$ be an IID Bernoulli($\mu$) process, independent of $X$. Observe that each of $(s_i,j)$, $i \in \{1,2\},j \in \{0,1\}$, is a renewal state for $(X,W)$. Observe also that the return time of $(s_1,1)$ in this process has the law of $T^*_\mu$.
    Since $X$ is mixing, $T$ is not supported on a proper subgroup of $\Z$.
    By \cref{lem:exp-geom}, $T^*_\mu$ has regular exponential tail when $\mu$ is small enough.

    Now let $Y'$ be the $\{1,2,*\}$-valued block factor of $(X,W)$ defined by $Y'_n := f(X_n,W_n)$, where $f(s_1,1) := 1$, $f(s_2,0)=f(s_2,1) := 2$, and $f(x,i) := *$ otherwise. It is not hard to see that both 1 and 2 are renewal states for $Y'$, and that their distributions are the same as those of $(s_1,1)$ in $(X,W)$ and of $s_2$ in $X$, respectively. Since $h(Y') \le h(X') + h(W) < h(X)$ holds provided that $\mu$ is small enough, by independently splitting the state $*$ in $Y'$, we obtain a process $Y$ of equal entropy to $X$. Since $X$ and $Y$ have a renewal state with the same distribution, they are finitarily isomorphic. Since $1$ is a renewal state for $Y$ whose return time has regular exponential tail, the proof is complete in this case.

    It remains to explain the reduction from the general case to the special case considered above.
    Consider the $k$-stringing $X^k$ of $X$, which is of course finitarily isomorphic to $X$.
    Observe that every state of $X^k$ has probability tending to zero as $k \to \infty$.
    Also, every state of $X^k$ whose first coordinate is $s_1$ is a renewal state. Thus, when $k$ is large enough, there exist two distinct renewal states $t_1$ and $t_2$ such that the entropy of $((\1_{\{X^k_n=t_1\}},\1_{\{X^k_n=t_2\}}))_{n \in \Z}$ is arbitrarily small. Since every renewal state in $X^k$ has exponential return time, this completes the reduction step.
\end{proof}

We now turn to the proof of \cref{lem:fin-iso2}, which is an adaptation of the ideas from~\cite{keane1979finitary,akcoglu1979finitary} (see also~\cite{shea2009finitary}) for Markov chains. 

\begin{proof}[Proof of \cref{lem:fin-iso2}]
    Let $Y'$ be a finite-state ergodic renewal process having a renewal state $s$ whose return time $T$ has semi-regular exponential tail. By sending all states other than $s$ to a new common state, and then independently splitting this common state, we may obtain a finite-state ergodic renewal process $Y$ of equal entropy to $Y'$ such that $s$ is a renewal state with the same distribution as in $Y'$. In particular, $Y$ and $Y'$ are finitarily isomorphic. The advantage of $Y$ over $Y'$ is that $\P(Y_{n+1} = s \mid Y_n,Y_{n-1},\dots) = f(Z_n+1)$ depends only on $Z_n := \min\{k \ge 0 : Y_{n-k}=s\}$, the time since the last occurrence of $s$, where
    \[ f(k) := \P(T=k \mid T \ge k) .\]
    Clearly, the process $(Y,Z)$ is clearly finitarily isomorphic to $Y$, and $Y_n=s$ if and only if $Z_n=0$. Note that $Z$ is a Markov chain on $\{0,1,\dots\}$ whose only allowed transitions are $k \mapsto 0$ and $k \mapsto k+1$, which have probabilities $f(k+1)$ and $1 - f(k+1)$, respectively.
    Note that the transition $k \mapsto k+1$ has positive probability for all $k$, since $T$ is unbounded.

    Let $k_0 \ge 1$ and let $(\alpha_k)_{k=0}^\infty$ be a sequence of numbers in $[0,1]$ such that $\alpha_k \propto \P(Z_{k_0}=0 \mid Z_0=k)^{-1}$.
    To see that this is possible, first note that by the definition of semi-regular exponential tail, $\liminf_{k \to \infty} f(k) > 0$ and $\limsup_{k \to \infty} f(k) < 1$, so that there exists $k_0 \ge 1$ and $c \in (0,\frac12)$ such that $f(k) \in [c,1-c]$ for all $k \ge k_0$.
    Note next that $\P(Z_{k_0}=0 \mid Z_0=k) \ge \P(T=k+k_0 \mid T \ge k+1) \ge c^{k_0}$.
    Thus, we may take $\alpha_k := c^{k_0} \cdot \P(Z_{k_0}=0 \mid Z_0=k)^{-1} \le 1$.

    Let $U$ be an IID process of uniform random variables, independent of $Z$.
    Let $W$ be an IID Bernoulli($\eps$) process, independent of $(Z,U)$, with $\eps$ small, but positive, to be determined later.
    The process $(Y,Z,U,W)$ takes values in $A \times \N \times [0,1] \times \{0,1\}$, where $A$ is the state space for $Y$. Let $X := (Y,Z,U,W)^{k_0+1}$ be the $(k_0+1)$-stringing of $(Y,Z,U,W)$.
    For each $k \ge 0$, consider the set of states $b(k)$ in $X$ given by
    \[ b(n) := \big\{ (a,z,u,w) \in (A \times \N \times [0,1] \times \{0,1\})^{\{0,\dots,k_0\}} : z_0 = k,~ z_{k_0}=0,~u_0<\alpha_k,~ w=(0,\dots,0,1) \big\} .\]
    Let $b := \bigcup_{k \ge 0} b(k)$.
    Note that $\P(X_0 \in b) \le \P(W_{k_0}=1) = \eps$.
    We claim that for any $n \ge 1$,
    \begin{equation}\label{eq:fin-iso-k-indep}
     \text{the events $\{ X_0 \in b \}$ and $\{ X_n \in b \}$ are } \begin{cases} \text{disjoint} &\text{if }n \le k_0\\ \text{independent} &\text{if }n > k_0 \end{cases}.
    \end{equation}
    The former is obvious since $(0,\dots,0,1)$ has no self overlaps. We proceed to show the latter.

    We first show that $Z_0$ is independent of $\{ X_0 \in b \}$.
    Of course, $Z_0=k$ on the event $\{ X_0 \in b(k) \}$, and the union of the latter events over $k \ge 0$ is $\{ X_0 \in b \}$. Thus, it suffices to show that the probability of $\{ X_0 \in b(k) \}$ is proportional to $\P(Z_0=k)$.
    Since $U$, $W$ and $Z$ are independent, this is the same as $\P(Z_0=k,Z_{k_0}=0) \cdot \P(U_0<\alpha_k) \propto \P(Z_0=k)$, which in turn is equivalent to $\P(Z_{k_0}=0 \mid Z_0=k) \propto \P(U_0<\alpha_k)^{-1} = \alpha_k^{-1}$, which holds by our choice of $k_0$ and $(\alpha_k)$.

    Fix $n>k_0$ and let us show that $E:=\{ X_0 \in b \}$ and $F:=\{ X_n \in b \}$ are independent. Denote $X^+ := (X_m)_{m \ge n}$ and $X^- := (X_m)_{m < n}$.
    Observe that the conditional law of $X^-$ given $X^+$ depends only on $Z_n$. Write $\pi_k$ for this law given that $Z_n=k$. Then the unconditional law of $X^-$ is $\sum_k \P(Z_n=k)\pi_k$.
    On the other hand, since $F$ is measurable with respect to $X^+$, the conditional law of $X^-$ given $F$ is $\sum_k \P(Z_n=k \mid F) \pi_k$. But we have seen that $Z_n$ is independent of $F$, so that this equals $\sum_k \P(Z_n=k)\pi_k$.
    Thus, the conditional law of $X^-$ given $F$ is the same as its unconditional law, and hence, $X^-$ and $F$ are independent. Since $E$ is measurable with respect to $X^-$, we conclude that $E$ and $F$ are independent.

    We now show that $(Y,Z)$ is finitarily isomorphic to an IID process.
    Fix some $a \in A \setminus \{s\}$, and note that the distribution of the state $t := ((a,1),\dots,(a,k_0+1))$ in $(Y,Z)^{k_0+1}$ has entropy less than $h(Y,Z)$.
    Define the three-valued process $X'$ as the 0-block factor of $X$ given by $X'_n=1$ if $X_n \in b$, $X'_n=2$ if $(Y,Z)^{k_0+1}_n=t$, and $X'_n=*$ otherwise. Since $\P(X_0 \in b) \le \eps$, this process has entropy less than that of $(Y,Z)$ when $\eps$ is small enough.
    Note that $X'_0=1$ forces $X'_1,\dots,X'_{k_0}=*$. It follows that both 1 and 2 are renewal states in $X'$ (see~\cite{keane1979finitary} for a similar argument).
    By independently splitting the state $*$ in $X'$, we may get a process $X''$ with the same entropy as $(Y,Z)$, for which 1 and 2 are still renewal states. Since the distribution of 2 in $X''$ is the same as that of $t$ in $(Y,Z)^{k_0+1}$, these two processes are finitarily isomorphic.
    Finally, a key feature of this construction due to~\eqref{eq:fin-iso-k-indep} is that the distribution of 1 in $X''$ satisfies that $\P(X''_n=1 \mid X''_0=1)=\P(X''_n=1) \cdot \1_{\{n>k_0\}}$ for all $n \ge 1$.
    On the other hand, this also characterizes the distribution of $0\cdots 01$ in the $(k_0+1)$-stringing of an IID process. Thus, $X''$ is finitarily isomorphic to an IID process. This shows that $(Y,Z)$ is also finitarily isomorphic to an IID process.
\end{proof}

\bibliographystyle{plain}
\bibliography{library}

\begin{thebibliography}{10}

\bibitem{akcoglu1979finitary}
M~A Akcoglu, Andr{\'e}s del Junco, and M~Rahe.
\newblock Finitary codes between {M}arkov processes.
\newblock {\em Zeitschrift f{\"u}r Wahrscheinlichkeitstheorie und Verwandte
  Gebiete}, 47(3):305--314, 1979.

\bibitem{angel2021markov}
Omer Angel and Yinon Spinka.
\newblock Markov chains with exponential return times are finitary.
\newblock {\em Ergodic Theory and Dynamical Systems}, 41(10):2918--2926, 2021.

\bibitem{keane1979bernoulli}
Michael Keane and Meir Smorodinsky.
\newblock Bernoulli schemes of the same entropy are finitarily isomorphic.
\newblock {\em Annals of Mathematics}, 109(2):397--406, 1979.

\bibitem{keane1979finitary}
Michael Keane and Meir Smorodinsky.
\newblock Finitary isomorphisms of irreducible {M}arkov shifts.
\newblock {\em Israel Journal of Mathematics}, 34(4):281--286, 1979.

\bibitem{meshalkin1959case}
LD~Meshalkin.
\newblock A case of isomorphism of {B}ernoulli schemes.
\newblock {\em DOKLADY AKADEMII NAUK SSSR}, 128(1):41--44, 1959.

\bibitem{meyerovitch2024entropy}
Tom Meyerovitch and Yinon Spinka.
\newblock Entropy-efficient finitary codings.
\newblock {\em Journal of Modern Dynamics}, 20:1--49, 2024.

\bibitem{ornstein1970bernoulli}
Donald Ornstein.
\newblock Bernoulli shifts with the same entropy are isomorphic.
\newblock {\em Advances in Mathematics}, 4(3):337--352, 1970.

\bibitem{petit1982deux}
B~Petit.
\newblock Deux sch{\'e}mas de bernoulli d'alphabet d{\'e}nombrable et
  d'entropie infinie sont finitairement isomorphes.
\newblock {\em Zeitschrift f{\"u}r Wahrscheinlichkeitstheorie und Verwandte
  Gebiete}, 59(2):161--168, 1982.

\bibitem{rudolph1981characterization}
Daniel~J Rudolph.
\newblock A characterization of those processes finitarily isomorphic to a
  {B}ernoulli shift.
\newblock In {\em Ergodic Theory and Dynamical Systems I}, pages 1--64.
  Springer, 1981.

\bibitem{rudolph1982mixing}
Daniel~J Rudolph.
\newblock A mixing {M}arkov chain with exponentially decaying return times is
  finitarily {B}ernoulli.
\newblock {\em Ergodic Theory and Dynamical Systems}, 2(1):85--97, 1982.

\bibitem{shea2009finitary}
Stephen~M Shea.
\newblock Finitary isomorphism of some renewal processes to {B}ernoulli
  schemes.
\newblock {\em Indagationes Mathematicae}, 20(3):463--476, 2009.

\end{thebibliography}

\end{document}